\documentclass[11pt,a4paper]{article}
\usepackage{amsfonts,amsgen,amsmath,amstext,amsbsy,amsopn,amsthm,amsfonts,amssymb,amscd}
\usepackage{epsf,epsfig}
\usepackage{float}
\usepackage{ebezier,eepic}
\usepackage{color}
\usepackage{tikz}
\usepackage{multirow}
\usepackage{graphicx}
\usepackage{subfigure}
\usepackage{epstopdf}
\newtheorem{thm}{Theorem}[section]

\newtheorem{prop}[thm]{Proposition}
\newtheorem{lem}[thm]{Lemma}
\newtheorem{false statement}{False statement}

\theoremstyle{definition}
\newtheorem{defn}[thm]{Definition}
\newtheorem{claim}{Claim}

\newtheorem{conj}[thm]{Conjecture}

\makeatletter \@addtoreset{equation}{section}

\def\hf{\mathcal{F}}

\title{\bf\Large Counterexamples to Gerbner's Conjecture on Stability of Maximal $F$-free Graphs}

\author{
Jian Wang\thanks{Department of Mathematics,
Taiyuan University of Technology, Taiyuan 030024, P.~R.~China. E-mail: {\tt  wangjian01@tyut.edu.cn}. Research supported by NSFC No.11701407.}~~~~
Shipeng Wang\thanks{Department of Mathematics, Jiangsu University, Zhenjiang, Jiangsu 212013, P.~R. China. E-mail: {\tt  spwang22@yahoo.com}. Research supported by NSFC No.12001242.}~~~~
Weihua Yang\thanks{Department of Mathematics, Taiyuan University of Technology, Taiyuan 030024, P.~R. China. E-mail: {\tt  yangweihua@tyut.edu.cn}. Research supported by NSFC No.11671296.}~~~~
}

\begin{document}

\date{}

\maketitle

\begin{abstract}
Let $F$ be an $(r+1)$-color critical graph with $r\geq 2$, that is, $\chi(F)=r+1$ and
there is an edge $e$ in $F$ such that $\chi(F-e)=r$. Gerbner recently conjectured that every  $n$-vertex maximal $F$-free graph  with at least $(1-\frac{1}{r})\frac{n^2}{2}- o(n^{\frac{r+1}{r}})$ edges contains an induced complete $r$-partite graph on $n-o(n)$ vertices. Let $F_{s,k}$ be a graph obtained from $s$ copies of $C_{2k+1}$ by sharing a common edge. In this paper, we show that for all $k\geq 2$ if $G$ is an $n$-vertex maximal $F_{s,k}$-free graph with at least $n^{2}/4 - o(n^{\frac{s+2}{s+1}})$ edges, then $G$ contains an induced complete bipartite graph on $n-o(n)$ vertices. We also show that it is best possible. This disproves Gerbner's conjecture for $r=2$.
\end{abstract}

\medskip

\section{Introduction}

A graph is called {\it $F$-free} if it does not contain $F$ as a subgraph. The {\it extremal number} $ex(n,F)$ is defined as the  maximum number of edges in an $F$-free $n$-vertex graph. Let $T_r(n)$ be the complete $r$-partite graph on $n$ vertices with partition classes of size $\lfloor \frac{n}{r}\rfloor$ or $\lceil \frac{n}{r}\rceil$ and let $t_r(n)$ be the number of edges in $T_r(n)$. The classical Tur\'{a}n Theorem~\cite{turan 1941} shows that $ex(n,K_{r+1})=t_r(n)$ and $T_r(n)$ is the unique graph attaining it. Since then the problem of determining $ex(n,F)$ becomes a central topic in extremal graph theory, which has been extensively studied.

In the past decades, many stability extensions to extremal problems were also well-studied. The stability phenomenon is that if an $F$-free graph is ``close" to extremal in the number of edges, then it must be ``close" to the extremal graph in its structure. The
famous stability theorem of Erd\H{o}s and Simonovits \cite{erdos1966,sim1968} implies the following: if $G$ is a $K_{r+1}$-free graph with $t_r(n)-o(n^2)$ edges, then $G$  can be made into a copy of $T_r(n)$ by adding or deleting $o(n^2)$ edges.

A graph $G$ is called {\it maximal $F$-free} if it is $F$-free and
the addition of any edge in the complement $\overline{G}$ creates a copy of $F$. Tyomkyn and Uzzel \cite{tyomkyn2015} considered a different kind of stability problems: when can one guarantee an
`almost spanning' complete $r$-partite subgraph in a maximal $K_{r+1}$-free graph $G$ with $t_r(n)-o(n^2)$ edges?
They showed that every maximal $K_{4}$-free graph $G$ with $t_3(n)-cn$ edges contains a complete $3$-partite subgraph on $(1-o(1))n$ vertices. Popielarz, Sahasrabudhe and Snyder \cite{popie2018} completely answered this question.

\begin{thm}[\cite{popie2018}]\label{PSSthm}
Let $r\geq 2$ be an integer. Every maximal $K_{r+1}$-free on $n$ vertices with at least $t_r(n)-o(n^{\frac{r+1}{r}})$ edges contains an induced complete $r$-partite subgraph on $(1-o(1))n$ vertices.
\end{thm}

Let $f(F,n,m)$ be the maximum integer $t$ such that every maximal $F$-free graph with at least $ex(n,F)-t$ edges contains an induced complete $(\chi(F)-1)$-partite subgraph on $n-m$ vertices. Popielarz, Sahasrabudhe and Snyder \cite{popie2018} give constructions to show that $f(K_{r+1},n,o(n))=o(n^{\frac{r+1}{r}})$. In \cite{WWYY}, Theorem \ref{PSSthm} was extended to maximal $C_{2k+1}$-free graphs.

\begin{thm}[\cite{WWYY}]\label{WWYYthm}
For every $k\geq 1$, $f(C_{2k+1},n,o(n))=o(n^{\frac{3}{2}})$.
\end{thm}

We say that a graph $F$ is {\it $(r+1)$-color-critical}, if $\chi(F)=r+1$ but
there is an edge $e$ in it such that $\chi(F-e)=r$. Recently, Gerbner proposed the following conjecture.

\begin{conj}[\cite{gerbner}]\label{GerbnerConj}
Let $r\geq 2$ be an integer and $F$ be an $(r+1)$-color critical graph. Then  $f(F,n,o(n))\geq o(n^{\frac{r+1}{r}})$.
\end{conj}

He verified Conjecture \ref{GerbnerConj} for some special 3-color-critical graphs.

\begin{thm}[\cite{gerbner}]\label{Gerbnerthm}
Let $F$ be a 3-color-critical graph in which
every edge has a vertex that is contained in a triangle. Then  $f(F,n,o(n))\geq o(n^{\frac{3}{2}})$.
\end{thm}

Let $F_{s,k}$  be a graph obtained from $s$ copies of $C_{2k+1}$ by sharing a common edge. It is easy to see that $F_{s,k}$ is 3-color-critical. Note that each vertex of $F_{s,1}$ is contained in a triangle, and so $f(F_{s,1},n,o(n))\geq o(n^{\frac{3}{2}})$ by Theorem \ref{Gerbnerthm}. Actually, one can show that  $f(F_{s,1},n,o(n))= o(n^{\frac{3}{2}})$ by a similar construction in \cite{popie2018}. Since $F_{1,k}=C_{2k+1}$, $f(F_{1,k},n,o(n))$ has been determined in Theorem \ref{WWYYthm}.

In this paper, we extend the two results above and determine $f(F_{s,k},n,o(n))$ for all $k\geq 2$ and $s\geq 2$, and this disproves Conjecture \ref{GerbnerConj} for $r=2$.

\begin{thm}\label{thm-main3}
For $k\geq 2$ and $s\geq 2$,
\begin{equation*}
f(F_{s,k},n,o(n))=o(n^{\frac{s+2}{s+1}}).
\end{equation*}
\end{thm}

We prove Theorem \ref{thm-main3} by the following two lemmas.

\begin{lem}\label{thm-main2}
For $k,s\geq 2$, $0\leq \alpha \leq \frac{1}{2}$ and $n\geq \frac{8k^2s^2}{\alpha}$, there is a maximal $F_{s,k}$-free graph $G$ with $e(G) \geq \frac{n^2}{4}-2ks\alpha n^{\frac{s+2}{s+1}}$ such that any induced complete bipartite subgraph of $G$ has at most $(1-\alpha^s)n$ vertices.
\end{lem}

\begin{lem}\label{thm-main}
Let $k,s\geq 2$. For sufficiently large $n$ and $0<\alpha<1$, if $G$ is an $n$-vertex maximal $F_{s,k}$-free graph with at least $n^{2}/4 - \alpha n^{\frac{s+2}{s+1}}$ edges, then $G$ contains an induced complete bipartite graph on $n-4\cdot (12sk)^{s+3}\alpha n$ vertices.
\end{lem}

In the rest of the paper, we prove Lemma \ref{thm-main2} in Section 2 and prove Lemma \ref{thm-main} in Section 3. We follow standard notation throughout.
Let $G$ be a graph. Denote by $\bar{G}$ the complement of $G$. For $v\in V(G)$, we use $N_G(v)$ to denote the set of neighbors of $v$ in $G$ and  let $\deg_G(v)=|N_G(v)|$.  Denote by $\delta(G)$ the minimum degree of $G$. Let $S$  be a subset of $V(G)$. We use $N_G(v,S)$ to denote the set of neighbors of $v$ in $S$ and let $\deg_G(v,S)=|N_G(v,S)|$. Denote by $G[S]$ and $G-S$ the subgraphs of $G$ induced by $S$ and $V(G)\setminus S$, respectively. When $S = \{v\}$, we simply write $G-v$ for $G-\{v\}$ .  Denote by $e_G(S)$ the number of edges of $G$ with both ends in $S$. For $xy\in E(\bar{G})$, let $G+xy$ be the graph obtained from $G$ by adding  $xy$. For $xy\in E(G)$, let $G-xy$ be the graph obtained from $G$ by deleting  $xy$. For any two disjoint subsets $X,Y$ of $V(G)$, let $G[X,Y]$ denote the bipartite subgraph of $G$ with the partite sets $X,Y$ and the edge set $$\{xy\in E(G)\colon x\in X \mbox{ and } y\in Y\}.$$
Denote by $e_G(X,Y)$ the number of edges in $G[X,Y]$. We also use $E[X,Y]$ and $\bar{E}[X,Y]$ to denote the edge set of $G[X,Y]$ and $\bar{G}[X,Y]$, respectively. We often omit the subscript when the underlying graph is clear. We also omit the floor and ceiling signs where they do not affect the arguments.

\section{Proof of Lemma \ref{thm-main2}}

In this section, we give a construction to show that for every small $\varepsilon >0$, there is a maximal $F_{s,k}$-free graph with at least $\frac{n^2}{4}-\varepsilon n^{\frac{s+2}{s+1}}$ edges, from which a positive fraction of vertices  has to be deleted to obtain an induced complete bipartite subgraph. First we introduce a function about the $t$-ary representations of positive integers, which will be used in our construction.

\begin{defn}
For every integer $s,t\geq 2$ and  $x\in [0,t^s-1]$, $x$ can be expressed uniquely as follows:
\[
x=q_{s-1}t^{s-1} +q_{s-2}t^{s-2}+\ldots+q_1t +q_0,
\]
where $q_0,q_1,\ldots,q_{s-1}\in [0,t-1]$.
We define the function $b_{t,s}(x,p)=q_p$ for $p=0,1,\ldots,s-1$.
\end{defn}

We construct a class of graphs, which contains the desired graph.

\begin{defn}\label{defn-1}
Given  $k\geq 2$, $s\geq 2$, $0<\alpha<\frac{1}{2}$ and $n\geq \frac{8k^2s^2}{\alpha}$. Let
$t = \alpha n^{\frac{1}{s+1}}$ and let $\mathcal{G}_{s,k,\alpha}(n)$ be a class of graphs as follows. A graph $G$ on $n$ vertices is in $\mathcal{G}_{s,k,\alpha}(n)$ if $V(G)$ can be partitioned into subsets
\[
X_0,\ldots,X_{t^s-1},X_{t^s}, Y_0,\ldots,Y_{t^s-1},Y_{t^s}
 \]
 and
 \[
 Z_{0,0},\ldots,Z_{0,t-1}, Z_{1,0},\ldots,Z_{1,t-1}, \ldots, Z_{s-1,0},\ldots,Z_{s-1,t-1}
\]
such that:

\begin{itemize}
  \item[(i)] For each $p=0,\ldots,s-1$ and $q=0,\ldots,t-1$, $|Z_{p,q}|=2k-1$ and $G[Z_{p,q}]$ contains a path of length $2k-2$, say $z_{p,q}^1z_{p,q}^2\ldots z_{p,q}^{2k-1}$.
  \item[(ii)] For each $i=0,1,\ldots,t^s-1$,
  \[
  |X_i|=|Y_i| = n^{\frac{1}{s+1}}.
  \]
  and $X_{t^s},Y_{t^s}$  is a balanced partition of
  \[
  V(G)\setminus \bigcup_{0\leq i\leq t^s-1}(X_i\cup Y_i) \setminus \bigcup_{\substack{0\leq p\leq s-1\\[2pt]0\leq q\leq t-1}} Z_{p,q}.
  \]
  \item[(iii)] For each $i=0,\ldots,t^s-1$, $G[X_i,Y_i]$ is empty and $G[X_{t^s},Y_{t^s}]$ is complete. For each $i,j\in \{0,\ldots ,t^s\}$ with $i\neq j$, $G[X_i,Y_j]$ is complete.
  \item[(iv)] Let $I(t,s,p,q)=\{i\in [0,t^s-1]:b_{t,s}(i,p)=q\}$. For each $p= 0,\ldots,s-1$ and each $q=0,\ldots,t-1$, $z_{p,q}^1$  is adjacent to every vertex in
  \[
      \bigcup_{i\in I(t,s,p,q)}X_i,
  \]
  and    $z_{p,q}^{2k-1}$ is adjacent to every vertex in
  \[
\bigcup_{i\in I(t,s,p,q)}Y_i.
  \]
\end{itemize}
\end{defn}

When refer to vertex classes of a graph in $\mathcal{G}_{s,k,\alpha}$, we use $X,Y,Z_p$ and $Z$ to denote
\[
\bigcup_{0\leq i\leq t^s}X_i,\ \bigcup_{0\leq i\leq t^s}Y_i,\ \bigcup_{0\leq q\leq t-1}Z_{p,q} \mbox{ and } \bigcup_{\substack{0\leq p\leq s-1\\[2pt]0\leq q\leq t-1}}Z_{p,q},
\]
respectively.

\begin{prop}\label{prop-2-3}
If $G$ is the graph in $\mathcal{G}_{s,k,\alpha}(n)$ with minimum number of edges, then $G$ is $F_{s,k}$-free.
\end{prop}

\begin{proof}
Suppose not, let $H$ be a copy of $F_{s,k}$ in $G$. By Definition \ref{defn-1} (i), $G[Z_{p,q}]$ is a path of length $2k-2$ for $p=0,1,\ldots,s-1$ and $q=0,1,\ldots,t-1$. Note that $z_{p,q}^{k}$ is the middle vertex on the path $G[Z_{p,q}]$.  Let
\begin{equation*}
\left\{
\begin{array}{l}Z_{p,q}^1= \{z_{p,q}^r\colon r< k \mbox{ and $r$ is odd}\} \cup \{z_{p,q}^r\colon r> k \mbox{ and $r$ is even}\},\\[6pt]
Z_{p,q}^2=\{z_{p,q}^r\colon r< k \mbox{ and $r$ is even}\} \cup \{z_{p,q}^r\colon r> k \mbox{ and $r$ is odd}\}.
\end{array}\right.
\end{equation*}
Clearly, $Z_{p,q} = Z_{p,q}^1\cup Z_{p,q}^2 \cup \{z_{p,q}^{k}\}$. By Definition \ref{defn-1} (iv), $z_{p,q}^1$ is not adjacent to any vertex in $Y$ and $z_{p,q}^{2k-1}$ is not adjacent to any vertex in $X$.  It follows that both $X\cup Z_{p,q}^2$ and $Y\cup Z_{p,q}^1$ are independent sets of $G$. Let
\[
Z^0 =\{z_{p,q}^{k}\colon 0\leq p\leq s-1, 0\leq q\leq t-1\},\ Z^1 = \bigcup_{\substack {0\leq p\leq s-1\\[2pt] 0\leq q\leq t-1}} Z_{p,q}^1 \mbox{ and } Z^2 = \bigcup_{\substack {0\leq p\leq s-1\\[2pt] 0\leq q\leq t-1}} Z_{p,q}^2.
\]
Then $Z^0$, $X\cup Z^2$ and $Y\cup Z^1$ are all independent sets of $G$. Let $xy$ be the common edge of $s$ cycles in $H$, and let $P^0,P^1,\ldots,P^{s-1}$ be $s$ paths of $H-xy$. Since $\deg_H(x)= s+1> 2$,  $\deg_H(y)= s+1> 2$ and $\deg_G(z_{p,q}^{k})=2$ for every $p\in [0,s-1]$ and $q\in [0,t-1]$, we have $\{x,y\}\cap Z^0=\emptyset$. Since $G-Z_0$ is bipartite and $P^i+xy$ is an odd cycle for $i=0,1,\ldots,s-1$, we see that $|V(P^i)\cap Z^0|\geq 1$, and let $z_{p_i,q_i}^{k}\in V(P^i)\cap Z^0$. By Definition \ref{defn-1} (i),  $G[Z_{p_i,q_i}]=z_{p_i,q_i}^1z_{p_i,q_i}^2\ldots z_{p_i,q_i}^{2k-1}$ is a subpath of $P^i$.  Then there are exactly two vertices on  $P^i+xy$ that are not in $Z_{p_i,q_i}$. By Definition \ref{defn-1} (iv), all the neighbors of $z_{p_i,q_i}^1$ except $z_{p_i,q_i}^2$ are in $X\setminus X_{t^s}$ and all the neighbors of $z_{p_i,q_i}^{2k-1}$ except $z_{p_i,q_i}^{2k-2}$ are in $Y\setminus Y_{t^s}$. Hence $V(P^i+xy)$ has  one vertex in $X\setminus X_{t^s}$,  one vertex in $Y\setminus Y_{t^s}$  and  all the other vertices in $Z_{p_i,q_i}$ for each $i=0,1,\ldots,s-1$.

For distinct $i, j\in \{0,1,\ldots,s-1\}$, we claim that $p_i\neq p_j$ or $q_i\neq q_j$. Otherwise, we have $V(P^i+xy)\cap V(P^j+xy)\supset Z_{p_i,q_i}$, implying that $|V(P^i+xy)\cap V(P^j+xy)|\geq 2k-1\geq 3$, a contradiction. (Note that here is the only place we use $k\geq 2$ in the proof and explain that the construction fails for $k=1$.) Thus $Z_{p_i,q_i}$ and $Z_{p_j,q_j}$ are disjoint, implying that  $Z_{p_i,q_i}, Z_{p_j,q_j}\subset V(H)\setminus\{x,y\}$. Moreover, one of $x, y$ is the common neighbor of $z_{p_i,q_i}^1$ and $z_{p_j,q_j}^1$ and the other is  the common neighbor of $z_{p_i,q_i}^{2k-1}$ and $z_{p_j,q_j}^{2k-1}$. By Definition \ref{defn-1} (iv),  $\{x,y\}\subset (X\setminus X_{t^s})\cup (Y\setminus Y_{t^s})$. Without loss of generality, we may assume that $x\in X_a$ and $y\in Y_b$ with $a,b\in [0,t^s-1]$. Since $xy$ is an edge in $H$, we have $a\neq b$.

Then $x$ is the common neighbor of $z^1_{p_0,q_0},z^1_{p_1,q_1},\ldots,z^1_{p_{s-1},q_{s-1}}$ and $y$ is the common neighbor of $z^{2k-1}_{p_0,q_0},z^{2k-1}_{p_1,q_1},\ldots,z^{2k-1}_{p_{s-1},q_{s-1}}$. By Definition \ref{defn-1} (iv), we have
$a\in  I(t,s,p_i,q_i)$  and $b\in  I(t,s,p_i,q_i)$, implying that
$b_{t,s}(a,p_i)=q_i$ and $b_{t,s}(b,p_i)=q_i$ for  $i=0,1,\ldots,s-1$. If $p_i= p_j$ for  some $i\neq j$, then $q_i=b_{t,s}(a,p_i)=b_{t,s}(a,p_j)=q_j$, contradicting the fact that $p_i\neq p_j$ or $q_i\neq q_j$. Thus $p_0,p_1,\ldots,p_{s-1}$ is a permutation of $\{0,1,\ldots,s-1\}$. Without loss of generality, we assume that $p_i=i$ for $i=0,1,\ldots,s-1$, then
\[
a=q_{s-1}t^{s-1} +q_{s-2}t^{s-2}+\ldots+q_1t +q_0=b,
\]
a contradiction. Therefore, $G$ is $F_{s,k}$-free.
\end{proof}

Now we are in a position to prove Lemma \ref{thm-main2}.

\begin{proof}[Proof of Lemma \ref{thm-main2}.]
By Proposition \ref{prop-2-3}, we may choose a maximal $F_{s,k}$-free graph  $G$ in $\mathcal{G}_{s,k,\alpha}(n)$.

\begin{claim}
Both $X$ and $Y$ are independent sets in $G$.
\end{claim}

\begin{proof}
By Definition \ref{defn-1} (iii), $G[X,Y_{t^s}]$ is complete bipartite. Note that
\begin{align*}
|X|>|Y_{t^s}|&=\frac{n-|Z|-|X\setminus X_{t^s}|-|Y\setminus Y_{t^s}|}{2}\\
&=\frac{n-st(2k-1)-2t^sn^{\frac{1}{s+1}}}{2}\\
&=\frac{n-s \alpha n^{\frac{1}{s+1}}(2k-1)-2(\alpha n^{\frac{1}{s+1}})^s n^{\frac{1}{s+1}}}{2}\\
&> \frac{n}{2}-ks \alpha n^{\frac{1}{s+1}}-\alpha^s n.
\end{align*}
Note that $s\geq 2, \alpha<\frac{1}{2}$ and $n\geq \frac{8k^2s^2}{\alpha}$. Then
\begin{align*}
|X|>|Y_{t^s}|> \frac{n}{2}-\frac{ks n^{\frac{1}{3}}}{2}-\frac{n}{4}\geq \frac{n^{\frac{1}{3}}}{4}(n^{\frac{2}{3}}-2ks ) \geq 2sk>|V(F_{s,k})|.
\end{align*}
If there is an edge $e$ in $G[X]$, then it is easy to find a copy of $F_{s,k}$ in $G[X\cup Y_{t^s}]$ because $F_{s,k}$ is 3-color-critical. Thus $X$ is an independent set of $G$. Similarly, $Y$ is an independent set of $G$.
\end{proof}

\begin{claim}
For $i=0,1,\ldots,t^s-1$, $G[X_i,Y_i]$ is empty.
\end{claim}

\begin{proof}
Suppose not, and let $xy$ be an edge with $x\in X_i$ and $y\in Y_i$. Assume that
\[
i = q_{s-1}t^{s-1} +q_{s-2}t^{s-2}+\ldots+q_1t +q_0.
\]
By Definition \ref{defn-1} (iv), $z^1_{0,q_0}, \ldots, z^1_{s-1,q_{s-1}}$ have a common neighbor $x$, and $z^{2k-1}_{0,q_0}, \ldots, z^{2k-1}_{s-1,q_{s-1}}$ have a common neighbor $y$. It follows that $G[Z_{0,q_0}\cup\ldots \cup Z_{s-1,q_{s-1}}\cup\{x,y\}]$ contains a copy of $F_{s,k}$, contradicting the fact that $G$ is $F_{s,k}$-free.
\end{proof}

By Claims 1 and 2, we have
\begin{align}\label{eq-4-1}
e(G[X\cup Y]) &= |X||Y|- \sum_{i=0}^{t^s-1} |X_i||Y_i|\nonumber\\
&\geq\left( \frac{ n-(2k-1)st}{2}\right)^2 - t^s \left( n^{\frac{1}{s+1}}\right)^2\nonumber\\
&=\left( \frac{ n-(2k-1)s\alpha n^{\frac{1}{s+1}}}{2}\right)^2 - \alpha^s n^{\frac{s}{s+1}} n^{\frac{2}{s+1}}\nonumber\\
&> \left( \frac{n}{2}- ks\alpha n^{\frac{1}{s+1}}\right)^2 - \alpha^s n^{\frac{s+2}{s+1}}\nonumber\\
&> \frac{n^2}{4}- k s\alpha n^{\frac{s+2}{s+1}}-\alpha^s n^{\frac{s+2}{s+1}}\nonumber\\
&> \frac{n^2}{4}-2 ks\alpha n^{\frac{s+2}{s+1}}.
\end{align}

In the following, we shall show that any induced complete bipartite subgraph of $G$ has at most $(1-\frac{\alpha}{4})n$ vertices. Assume that $H$ is a largest induced complete bipartite subgraph of $G$ with vertex classes $A$ and $B$. Note that each vertex of $X_i$ (or $Y_i$)
plays the same role in $G$. If there is a vertex in $X_i$ (or $Y_i$)  belongs to $V(H)$, then by the maximality of $H$, every vertex of $X_i$ (or $Y_i$) belongs to $V(H)$.

Suppose first that $X_{t^s}\cap( A\cup B)=\emptyset$ or $Y_{t^s}\cap( A\cup B)=\emptyset$. Then
\begin{align}\label{ineq-4-2}
|A|+|B| &\leq n-\min \{|X_{t^s}|,|Y_{t^s}|\}\nonumber\\
&= n -\frac{n-|Z|-|X\setminus X_{t^s}|-|Y\setminus Y_{t^s}|}{2}\nonumber\\
&= \frac{n+|Z|+|X\setminus X_{t^s}|+|Y\setminus Y_{t^s}|}{2}\nonumber\\
&=\frac{n+s \alpha n^{\frac{1}{s+1}}(2k-1)+2(\alpha n^{\frac{1}{s+1}})^s n^{\frac{1}{s+1}}}{2}\nonumber\\
&< \frac{n}{2}+ ks\alpha n^{\frac{1}{s+1}}+\alpha^s n\nonumber\\
&\leq  \left(1-\alpha^s\right)n.
\end{align}

Now suppose that $X_{t^s}\cap( A\cup B)\neq \emptyset$ and $Y_{t^s}\cap( A\cup B)\neq\emptyset$. Then $X_{t^s},Y_{t^s}\subset A\cup B$. Without loss of generality, we assume that $X_{t^s}\subset A$ and $Y_{t^s}\subset B$.
Since $G[X_i,Y_i]$ ($0\leq i\leq t^s-1$) is empty, and both $G[X_i,Y_{t^s}]$ and $G[X_{t^s},Y_i]$ are complete bipartite, it follows that at most one of $X_i$ and $Y_i$ is in $A\cup B$. Hence $H$ is missing at least $t^s$ of $X_0, \ldots, X_{t^s-1}, Y_0, \ldots, Y_{t^s-1}$ and so
\begin{align}\label{eq-4-4}
|A\cup B| &\leq n-t^s n^{\frac{1}{s+1}} = n- \alpha^s n = \left(1-\alpha^s\right)n.
\end{align}
This completes the proof.
\end{proof}

\section{Proof of Lemma \ref{thm-main}}

In this section, we prove a stability theorem for maximal $F_{s,k}$-free graphs. We say that a vertex of a graph $G$ is {\it color-critical}, if deleting
that vertex results in $G$ with smaller chromatic number. The following two results are needed.

\begin{lem}[\cite{gerbner}]\label{lem-3-1}
Let $F$ be a 3-chromatic graph with a color-critical vertex and $n$ be
sufficiently large. Let $\frac{20|V(F)|}{n}<\varepsilon<\frac{1}{11|V(F)|^2}$. If $G$ is an $n$-vertex $F$-free graph with $|E(G)|\geq ex(n,F)-\varepsilon n^2$,  then there is a bipartite subgraph $G'$ of $G$ with at least $(1-12|V(F)|\varepsilon )n$ vertices, at least $ex(n,F)-13|V(F)|\varepsilon n^2$ edges and minimum
degree at least $\left(\frac{1}{2}-\frac{1}{11|V(F)|}\right)n$ such that every vertex of $G'$ is adjacent in $G$ to at most $|V(F)|$ vertices in the same partite set of $G'$.
\end{lem}

\begin{thm}[\cite{sim1974}]\label{thm-3-2}
 Let $F$ be an $(r+1)$-color-critical graph. There exists an $n_0$ such that if $n > n_0$, then $ex(n,F)=t_r(n)$.
\end{thm}

We find a large induced bipartite subgraphs with useful structures in maximal $F_{s,k}$-free graphs by the following lemma.

\begin{lem}\label{lem-3-3}
Let $G$ be an $n$-vertex maximal $F_{s,k}$-free graph with at least $\frac{n^{2}}{4} - \varepsilon n^2$ edges and let $h=|V(F_{s,k})|$. Then there is a partition $(U, V, T)$ of $V(G)$ such that
\begin{itemize}
  \item[(i)] $\left(\frac{1}{2}-\frac{1}{10h}\right)n\leq |U|,|V|\leq \left(\frac{1}{2}+\frac{1}{10h}\right)n$ and $|T|\leq 30h^2\varepsilon n$;
  \item[(ii)] $G[U\cup V]$ is an induced bipartite subgraph of $G$ with partite sets $U,V$, minimum degree $\left(\frac{1}{2}-\frac{1}{10h}\right)n$ and at least $\frac{n^2}{4}-25h^2\varepsilon n^2$ edges;
  \item[(iii)] for every $x\in T$, if $x$ has neighbors in $U$(or $V$), then it has at least $h+1$ neighbors in $U$(or $V$).
\end{itemize}
\end{lem}

\begin{proof}
Since $F_{s,k}$ is 3-color-critical, by Theorem \ref{thm-3-2} we have $ ex(n,F_{s,k})=\lfloor\frac{n^2}{4}\rfloor$. Since $F_{s,k}$ has two critical vertices, by Lemma \ref{lem-3-1} there is a bipartite subgraph $G'$ of $G$ with at least $(1-12h\varepsilon)n$ vertices, at least $\frac{n^2}{4}-13h\varepsilon n^2$ edges and minimum degree at least $\left(\frac{1}{2}-\frac{1}{11h}\right)n$. Let $U_0,V_0$ be two partite sets of $G'$ and let $T_0=V(G)\setminus V(G')$. Clearly, $\left(\frac{1}{2}-\frac{1}{11h}\right)n\leq |U_0|,|V_0|\leq \left(\frac{1}{2}+\frac{1}{11h}\right)n$ and $|T_0|\leq 12h\varepsilon n$.

\begin{claim}
Both $U_0$ and $V_0$ are independent sets in $G$, that is, $G'$ is induced.
\end{claim}
\begin{proof}
 By contradiction, we may assume, without loss of generality, that $U_0$ is not an independent set. Then there is an edge $u_1u_2$ in $G[U_0]$. Since $\delta(G')\geq \left(\frac{1}{2}-\frac{1}{11h}\right)n$ and $|V_0|\leq \left(\frac{1}{2}+\frac{1}{11h}\right)n$, each $u_i (i=1,2$) has at most $\frac{2n}{11h}$ non-neighbors in $V_0$. It follows that $u_1,u_2$ have at least $\left(\frac{1}{2}-\frac{5}{11h}\right)n$ common neighbors in $V_0$. Let $V_0'$ be the set of the common neighbors of $u_1, u_2$ and $U_0'=U_0\setminus \{u_1,u_2\}$. By $\delta(G')\geq \left(\frac{1}{2}-\frac{1}{11h}\right)n$ and since $n$ is sufficiently large, we have
\[
e(U_0',V_0') \geq |V_0'| \left(\left(\frac{1}{2}-\frac{1}{11h}\right)n-2\right)> \left(\frac{1}{2}-\frac{5}{11h}\right)n\left(\frac{1}{2}-\frac{2}{11h}\right)n> \frac{n^2}{6}\geq\frac{(h+s-3)n}{2}.
\]
By Erd\H{o}s-Gallai theorem \cite{erdos-gallai}, there is a path $P$ on $h+s-1$ vertices in $G[U_0',V_0']$. We truncate $P$ into $s$ vertex-disjoint paths with endpoints in $V_0'$ and each of length $2k-2$. These paths together with $u_1,u_2$ form a copy of  $F_{s,k}$, a contradiction.
\end{proof}

Let $T=T_0$, $U=U_0$ and $V=V_0$. Now we remove a small amount of vertices from $U$ to $T$ by a greedy algorithm.
In each step, if there is a vertex $x\in T$ with $1\leq \deg(x,U)\leq h$, then we remove all the neighbors of $x$  from $U$ to $T$. If every vertex in $T$ either has at least $h+1$ neighbors or  no neighbors in $U$, then we stop.  By Claim 3, $U_0$ is an independent set, then each vertex added in $T$ has no neighbors in $U$. Moveover, if  all the neighbors of $x\in T_0$ have been  removed from $U$ to $T$, then $x$ has no neighbors in the updated $U$. Hence the algorithm will stop in at most $|T_0|$ steps.
Let $U'$ be the vertices removed from $U$ to $T$ by the algorithm. It follows that
\[
|U'|\leq h|T_0| \leq 12h^2\varepsilon n.
\]

Then  we remove a small amount of  vertices from $V$ to $T$ similarly. In each step, if there is a vertex $x\in T$ with $1\leq \deg(x,V)\leq h$, then we remove all the neighbors of $x$  from $V$ to $T$. Similarly, the algorithm will stop in at most $|T_0|+|U'|$ steps. Since  $\delta(G')\geq \left(\frac{1}{2}-\frac{1}{11h}\right)n$,  each $x\in U'$ has at least $\left(\frac{1}{2}-\frac{1}{11h}\right)n$ neighbors in $V_0$. It follows that each $x\in U'$ has at least $\left(\frac{1}{2}-\frac{1}{11h}\right)n-(|T_0|+|U'|)h\geq h+1$ neighbors in $V$ in the executing of the algorithm. That is, the neighbors of vertices in $U'$ will not be removed in the algorithm. Hence, the algorithm will stop in at most $|T_0|$ steps. Let $V'$ be the vertices removed from $V$ to $T$ by the algorithm. It follows that
\[
|V'|\leq h|T_0| \leq 12h^2\varepsilon n.
\]

Let $U,V,T$ be the resulting sets at the end of the algorithm. By Claim 3 and since $U\subset U_0, V\subset V_0$, both $U$ and $V$ are independent sets. Let $G''$ be the bipartite subgraph induced by $U$ and $V$. Since both $|U'|$ and $|V'|$ have size at most $12h^2\varepsilon n$, we have
\[
|T|\leq |T_0|+|U'|+|V'|\leq 12h\varepsilon n+24h^2\varepsilon n\leq 30h^2\varepsilon n,
\]
and
\begin{align*}
e(G'') &\geq e(G')-(|U'|+|V'|)\cdot \max\{|U_0|,|V_0|\} \\
&\geq \frac{n^2}{4}-13h\varepsilon n^2 - 24h^2\varepsilon n \left(\frac{1}{2}+\frac{1}{11h}\right)n\\
&\geq \frac{n^2}{4}-25h^2\varepsilon n^2,
\end{align*}
and
\[
\delta(G'') \geq \delta(G')- \max\{|U'|,|V'|\} \geq \left(\frac{1}{2}-\frac{1}{11h}\right)n -12h^2\varepsilon n \geq \left(\frac{1}{2}-\frac{1}{10h}\right)n.
\]
It follows that
\[
\left(\frac{1}{2}-\frac{1}{10h}\right)n\leq |U|,|V|\leq \left(\frac{1}{2}+\frac{1}{10h}\right)n.
\]
Moreover, for each $x\in T$, $x$ either has at least $h+1$ neighbors or  no neighbors in $U$, and $x$ either has at least $h+1$ neighbors or no neighbors in $V$. Thus the lemma holds.
\end{proof}

\begin{lem}\label{lem-3-4}
 Let $G$ be a bipartite graph with partite sets $U,V$ and let $W$ be a subset of $U\cup V$ with $|W|=h$. If $\left(\frac{1}{2}-\frac{1}{10h}\right)n\leq |U|,|V|\leq \left(\frac{1}{2}+\frac{1}{10h}\right)n$, $\delta(G)\geq\left(\frac{1}{2}-\frac{1}{10h}\right)n$ and  $n\geq 10h$, then the following holds.
 \begin{itemize}
  \item[(i)]  For every $u\in U, v\in V$ and every odd integer $l$ with $3\leq l\leq h$, there is a $uv$-path $P$ of length $l$ such that $(V(P)\setminus\{u,v\})\cap W=\emptyset$.
  \item[(ii)] For every $u,v\in U$ and every even integer $l$ with $2\leq l\leq h$, there is a $uv$-path $P$ of length $l$ such that $(V(P)\setminus\{u,v\})\cap W=\emptyset$.
 \end{itemize}
\end{lem}

\begin{proof}
For any $u\in U, v\in V$, let $A=N(v)\setminus (W\cup \{u\})$  and $B=N(u)\setminus (W\cup \{v\})$. Then
\[
\left(\frac{1}{2}-\frac{1}{10h}\right)n-h-1\leq |A|,|B|\leq \left(\frac{1}{2}+\frac{1}{10h}\right)n
\]
and the minimum degree of $G[A,B]$ is at least
\begin{align*}
\delta(G)-\max\left\{|U|-|A|,|V|-|B|\right\}
&\geq \left(\frac{1}{2}-\frac{1}{10h}\right)n- \left(\frac{n}{5h}+h+1\right)\\
&\geq\left(\frac{1}{2}-\frac{3}{10h}\right)n-h-1.
\end{align*}
It follows that
\begin{align*}
e(A,B)=\frac{1}{2}& \sum_{x\in A\cup B} \deg_{G[A,B]}(x)\\
&\geq \frac{1}{2}\left(\left(\frac{1}{2}-\frac{3}{10h}\right)n-h-1\right)(|A|+|B|)\\
&\geq \frac{1}{2}\left(\left(\frac{1}{2}-\frac{3}{10h}\right)10h-h-1\right)(|A|+|B|)\\
&> \frac{h}{2}(|A|+|B|)\\
&> \frac{(l-2)-1}{2}(|A|+|B|).
\end{align*}
For any odd integer $l$ with $3\leq l\leq h$, there is a path of length $l-2$ in $G[A,B]$ by Erd\H{o}s-Gallai Theorem \cite{erdos-gallai}, which together with $u,v$ is our desired path.

If $u,v\in U$, then let $A=U\setminus (W\cup\{u,v\})$ and $B=N(u)\cap N(v)\setminus W$. Clearly,
\[
|A|\geq \left(\frac{1}{2}-\frac{1}{10h}\right)n-h-2
\]
and
\begin{align*}
|B| &\geq |N(u)\cap N(v)|-h\\
&\geq |N(u)|+| N(v)|-|V|-h\\
&\geq 2\left(\frac{1}{2}-\frac{1}{10h}\right)n -\left(\frac{1}{2}+\frac{1}{10h}\right)n -h\\
&= \left(\frac{1}{2}-\frac{3}{10h}\right)n -h.
\end{align*}
The minimum degree of $G[A,B]$ is at least
\begin{align*}
\delta(G)-\max\left\{|U|-|A|,|V|-|B|\right\}
&\geq \left(\frac{1}{2}-\frac{1}{10h}\right)n- \left(\frac{2n}{5h}+h\right)\\
&\geq\left(\frac{1}{2}-\frac{1}{2h}\right)n-h.
\end{align*}
It follows that
\begin{align*}
e(A,B)=\frac{1}{2}& \sum_{x\in A\cup B} \deg_{G[A,B]}(x)\\
&\geq \frac{1}{2}\left(\left(\frac{1}{2}-\frac{1}{2h}\right)n-h\right)(|A|+|B|)\\
&\geq \frac{1}{2}\left(\left(\frac{1}{2}-\frac{1}{2h}\right)10h-h\right)(|A|+|B|)\\
&> \frac{h}{2}(|A|+|B|)\\
&> \frac{l-1}{2}(|A|+|B|).
\end{align*}
For any even integer $l$ with $2\leq l\leq h$, there is a path of length $l$ in $G[A,B]$ by Erd\H{o}s-Gallai Theorem \cite{erdos-gallai}, say $x_1x_2\ldots x_{l+1}$. If $x_1,x_{l+1}\in A$, then $ux_2\ldots x_lv$ is the desired path. If $x_1,x_{l+1}\in B$, then $ux_3\ldots x_{l+1}v$ is the desired path. This completes the proof.
\end{proof}

We need some definitions in the proof of Lemma \ref{thm-main}. Let $F,G$ be two graphs. A {\it homomorphism} from $F$ to $G$ is a mapping $\phi: V(F)\rightarrow V(G)$ with the property that $\{\phi(u),\phi(v)\}\in E(G)$ whenever $\{u,v\}\in E(F)$. A homomorphism from $F$ to $G$ is also called an {\it $F$-homomorphism} in $G$. If $\phi$ is injective, then $\phi$ is called an {\it injective homomorphism}. If  $\phi$ is both injective and surjective, then $\phi$ is called an {\it isomorphism}. Now we prove Lemma \ref{thm-main} by a delicate vertex-deletion process.

\begin{proof}[Proof of Lemma \ref{thm-main}.]
Let $G$ be an $n$-vertex maximal $F_{s,k}$-free graph with at least $\frac{n^{2}}{4} - \varepsilon n^2$ edges and let $h=|V(F_{s,k})|$. By Lemma \ref{lem-3-3}, there is a partition $(U, V, T)$ of $V(G)$ satisfying conditions (i), (ii) and (iii) of Lemma \ref{lem-3-3}. Let $G'=G[U,V]$.  We are left to delete vertices in $G'$ until the resulting graph is complete bipartite.

We write $F$ instead of $F_{s,k}$ for simplicity.   
For any  non-edge $xy$ of $G'$ with $x\in U$ and $y\in V$, $G+xy$ contains at least one copy of $F$ since $G$ is maximal $F$-free. Let $F_{xy}$ be one of such copies and let $\phi_{xy}$ be the isomorphism from $F$ to $F_{xy}$. Let
\[
\Omega = \{xy\colon x\in U,\ y\in V\mbox{ and } xy\notin E(G')\}.
\]
\begin{claim}\label{claim-4}
For each $xy\in \Omega$, $N_{F_{xy}}(x)\cap T\neq \emptyset$ and $N_{F_{xy}}(y)\cap T\neq \emptyset$.
\end{claim}

\begin{proof}
By contradiction, we may suppose that $N_{F_{xy}}(x)\cap T=\emptyset$ without loss of generality. Let $y_0=y,y_1,\ldots,y_p$ be the neighbors of $x$ in $F_{xy}$. Then $y_1,\ldots,y_p$ are all in $V$ because $U$ is an independent set. Since the maximum degree of $F_{xy}$ is  $s+1$, it follows that $p\leq s$. By Lemma \ref{lem-3-3} (ii), we have  $\delta(G')\geq \left(\frac{1}{2}-\frac{1}{10h}\right)n$ and $\left(\frac{1}{2}-\frac{1}{10h}\right)n\leq |U|\leq \left(\frac{1}{2}+\frac{1}{10h}\right)n$, then each $y_i$ ($i=0,1,\ldots,p$) has at most $\frac{n}{5h}$ non-neighbors in $U$. Note that
 \[
h= s(2k-1)+2\geq 3s+2 > 2s+3\geq 2p+3
\]
as $k\geq 2$ and $s\geq p$.
Therefore, the number of common neighbors of  $y_0,y_1,\ldots,y_p$ in $U$ is at least
\begin{align*}
\left(\frac{1}{2}-\frac{1}{10h}\right)n-(p+1)\frac{n}{5h}&\geq \left(\frac{1}{2}-\frac{2p+3}{10h}\right)n\\
&>  \left(\frac{1}{2}-\frac{1}{10}\right)n\\
&\geq \frac{2n}{5}\\
& >h.
\end{align*}
Thus, there is a vertex $x'\in U$ such that $x'\notin V(F_{xy})$ and $x'y_i\in E(G)$ for $i=0,1,\ldots,p$. Then by replacing $x$ with $x'$ in $F_{xy}$ we obtain a copy of $F$ in $G$, contradicting the fact that $G$ is $F$-free.
\end{proof}

Let $a,b$ be the vertices of degree $s+1$ in $F$ and let $ac_1^i \ldots c_{2k-1}^i b$ ($i=1,\ldots,s$) be those paths in $F-ab$. Now we partition $\Omega$ into three classes as follows:
\begin{equation*}
\left\{
\begin{array}{l}\Omega_1=\left\{xy\in \Omega\colon \phi_{xy}^{-1}(x),\phi_{xy}^{-1}(y)\in V(F)\setminus\{a,b\}\right\},\\[6pt]
\Omega_2=\left\{xy\in \Omega\colon  \{\phi_{xy}^{-1}(x),\phi_{xy}^{-1}(y)\}=\{a,b\}\right\},\\[6pt]
\Omega_3=\left\{xy\in \Omega\colon |\{\phi_{xy}^{-1}(x),\phi_{xy}^{-1}(y)\}\cap\{a,b\}|=1 \right\}.
\end{array}\right.
\end{equation*}

We delete a small amount of vertices from $U\cup V$ to destroy all non-edges in $\Omega$ in the following three steps.

{\noindent\bf Step 1.} We can find $U_1\subset U$ and $V_1\subset V$ such that $|U\setminus U_1|+|V\setminus V_1| \leq 160 h^3\varepsilon n^{\frac{3}{2}}$ and $\bar{E}[U_1,V_1] \cap \Omega_1=\emptyset$. That is, by deleting at most $160 h^3\varepsilon n^{\frac{3}{2}}$ vertices from $U\cup V$ we destroy all non-edges in $\Omega_1$.

\begin{proof}
If  $\Omega_1=\emptyset$, we have nothing to do. So assume that $\Omega_1\neq\emptyset$, then  there is a non-edge $xy$ in $\Omega_1$ with $x\in U$ and $y\in V$. By definition of $\Omega_1$, we see that both $x$ and $y$  have degree two in $F_{xy}$. By Claim \ref{claim-4},  $N_{F_{xy}}(x)\cap T\neq \emptyset$ and $N_{F_{xy}}(y)\cap T\neq \emptyset$. Let $x^*\in N_{F_{xy}}(x)\cap T$ and $y^*\in N_{F_{xy}}(y)\cap T$.  Then  $x^*\neq y^*$ since $F_{xy}$ is triangle-free.  Let $X=N_G(x^*,U)$, $Y=N_G(y^*,V)$ and let  $S$ be one of $X$ and $Y$ with smaller size.

For any edge $x'y'$ in $G$ with $x'\in X$ and $y'\in Y$, if $\{x',y'\}\cap V(F_{xy})=\emptyset$, then by replacing $x,y$ with $x',y'$ in $F_{xy}$ we obtain a copy of $F$ in $G$, a contradiction. Thus,
every edge in $G[X,Y]$ intersects $V(F_{xy})$, implying that $e(X,Y)\leq h(|X|+|Y|)$. Then
\[
e_{\bar{G}}(X,Y) = |X||Y| - e(X,Y) \geq |X||Y|-h(|X|+|Y|).
\]
Without loss of generality, we assume that $|X|\leq |Y|$, then $S = X$. If $|S|\geq 4h$, then
\begin{align*}
e_{\bar{G}}(X,Y)& \geq |S||Y|-h(|S|+|Y|)\\
& = |Y|(|S|-h)-h|S|\\
& \geq |S|^2-2h|S|\\
& \geq \frac{|S|^2}{2}>\frac{|S|^2}{16h^2}.
\end{align*}
If $|S|< 4h$, then since $xy$ is a non-edge of $G$ between $X$ and $Y$, we have
\[
e_{\bar{G}}(X,Y) \geq 1 >  \frac{|S|^2}{16h^2}.
\]
Thus, there are at least $\frac{|S|^2}{16h^2}$ non-edges between $X$ and $Y$. We delete vertices in $S$ from $U\cup V$ and let $U'=U\setminus S$ and $V'=V\setminus S$. If $\bar{E}[U',V'] \cap \Omega_1=\emptyset$, then we are done. Otherwise, there is another non-edge $xy$ in $\Omega_1$ with $x\in U', y\in V'$, and we delete another $S'$ from $U'\cup V'$ incidents with at least $\frac{|S'|^2}{16h^2}$ non-edges between $U'$ and $V'$. By deleting vertices greedily, we shall obtain a sequence of disjoint  sets $S_1,S_2,\ldots,S_l$ in $U\cup V$ such that $\bar{E}[U\setminus (S_1\cup \ldots \cup S_l),V\setminus (S_1\cup \ldots \cup S_l)] \cap \Omega_1=\emptyset$. In each step of the greedy algorithm, there is a $u\in T$ such that either $N(u)\cap U$ or $N(u)\cap V$ is deleted, implying that $l\leq 2|T|$.

By Lemma \ref{lem-3-3} (ii), $G[U,V]$ has at least $\frac{n^2}{4}-25h^2\varepsilon n^2$ edges. It follows that the number of non-edges between $U$ and $V$ is at most
\[
|U||V|-\left(\frac{n^2}{4}-25h^2\varepsilon n^2\right) \leq 25h^2\varepsilon n^2.
\]
Thus,
\begin{align}\label{ineq3-3}
\sum_{i=1}^l \frac{|S_i|^2}{16h^2}\leq 25h^2\varepsilon n^2.
\end{align}
By Cauchy-Schwarz inequality, we have
\begin{align}\label{ineq3-4}
\left(\sum_{i=1}^{l} |S_i|\right)^2 \leq \left(\sum_{i=1}^{l} |S_i|^2\right)l.
\end{align}
Note that  $|T|\leq 30h^2\varepsilon n$ from Lemma \ref{lem-3-3} (i).
By \eqref{ineq3-3}, \eqref{ineq3-4}  and $l\leq 2|T|$, we arrive at
\[
\left(\sum_{i=1}^{l} |S_i|\right)^2  \leq 16h^2 \cdot 25h^2\varepsilon n^2 l\leq  20^2 h^4\varepsilon n^2 \cdot 2|T| \leq  20^2 h^4\varepsilon n^2 \cdot 60h^2\varepsilon n.
\]
Let $U_1=U\setminus (S_1\cup \ldots \cup S_l)$ and $V_1=V\setminus (S_1\cup \ldots \cup S_l)$. Then
\[
|U\setminus U_1|+|V\setminus V_1| \leq \sum_{i=1}^{l} |S_i| \leq 160 h^3\varepsilon n^{\frac{3}{2}}
\]
and  Step 1 is finished.
\end{proof}

{\noindent \bf Step 2.}
We can find $U_2\subset U_1$ and $V_2\subset V_1$ such that $|U_1\setminus U_2|+|V_1\setminus V_2| \leq (6h)^{s+3}\varepsilon^{\frac{s+1}{2}} n^{\frac{s+2}{2}}$ and $\bar{E}[U_2,V_2] \cap (\Omega_1\cup\Omega_2)=\emptyset$. That is,  by deleting at most $(6h)^{s+2}\varepsilon^{\frac{s+1}{2}} n^{\frac{s+2}{2}}$ vertices from $U_1\cup V_1$ we destroy all non-edges in $\Omega_2$.

\begin{proof}
By Step 1, we see that $E[U_1,V_1] \cap \Omega_1=\emptyset$. Thus, we are left to delete vertices from $U_1\cup V_1$ to destroy all  non-edges in $\Omega_2\cap \bar{E}[U_1,V_1]$. If $\Omega_2\cap \bar{E}[U_1,V_1]=\emptyset$, we have nothing to do. So assume that $\Omega_2\cap \bar{E}[U_1,V_1]\neq\emptyset$, then  there is a non-edge $xy$ in $\Omega_2$ with $x\in U_1$ and $y\in V_1$. For $xy\in \Omega_2$, let
\[
\mathcal{F}_{xy}=\left\{F_{xy}\colon F_{xy} \mbox{ is a copy of } F \mbox{ in } G+xy \mbox{ with }\deg_{F_{xy}}(x)= \deg_{F_{xy}}(y)=s+1\right\}.
\]
Clearly, $\mathcal{F}_{xy}\neq \emptyset$.
\begin{claim}
There is an $F_{xy}\in \mathcal{F}_{xy}$ such that
\begin{itemize}
  \item[(i)] for every $u\in V(F_{xy})\setminus \left(T\cup \{x,y\}\right)$, $\deg_{F_{xy}}(u,T)\leq 1$;
  \item[(ii)] for every $uv\in E(F_{xy}-T- \{x,y\})$,  $\deg_{F_{xy}}(u,T)+\deg_{F_{xy}}(v,T)\leq 1$.
\end{itemize}
\end{claim}

\begin{proof}
For any  $F_{xy}\in \mathcal{F}_{xy}$,  let
\[
\theta_1(F_{xy}) =\left|\left\{u\in V(F_{xy})\setminus \left(T\cup \{x,y\}\right)\colon \deg_{F_{xy}}(u,T)=2\right\}\right|
\]
and
\[
\theta_2(F_{xy}) =\left|\left\{uv\in E(F_{xy}-T- \{x,y\})\colon \deg_{F_{xy}}(u,T)+\deg_{F_{xy}}(v,T)\geq 2\right\}\right|.
\]
We choose $F_{xy}$ from $\mathcal{F}_{xy}$ such that $\theta_1(F_{xy})+\theta_2(F_{xy})$ is minimized, and show that $\theta_1(F_{xy})=\theta_2(F_{xy})=0$ to finish the proof. Suppose first that $\theta_1(F_{xy})\geq 1$. Then there is a $u\in V(F_{xy})\setminus \left(T\cup \{x,y\}\right)$ such that $\deg_{F_{xy}}(u,T)=2$. Let $C= xy\ldots u_1^*uu_2^*\ldots x$ be the cycle in $F_{xy}$ with $u_1^*,u_2^*\in T$. Clearly, $C$ has length $2k+1$. Without loss of generality, we assume that $u\in U$.  If the path $uu_2^*\ldots x$ has even length $l$, then by Lemma \ref{lem-3-4} (ii) with $W=V(F_{xy})$ there is a $ux$-path $P$ of length $l$ in $G[U,V]$ such that $V(P)\cap V(F_{xy})=\{u,x\}$. By replacing $uu_2^*\ldots x$ from $F_{xy}$ with  $P$, we obtain a new copy $F_{xy}'$ of $F$ with $F_{xy}'\in \hf$ and $\theta_1(F_{xy}')<\theta_1(F_{xy})$, contradicting the choice of $F_{xy}$. If the path $uu_2^*\ldots x$ has odd length $l$, then by Lemma \ref{lem-3-4} (i) with $W=V(F_{xy})$ there is a $uy$-path $Q$ of length $2k-l$ in $G[U,V]$ such that $V(Q)\cap V(F_{xy})=\{u,y\}$. By replacing $y\ldots u_1^*u$ from $F_{xy}$ with  $Q$, we obtain a new copy $F_{xy}'$ of $F$ with $F_{xy}'\in \hf$ and $\theta_1(F_{xy}')<\theta_1(F_{xy})$, contradicting the choice of $F_{xy}$.

Suppose next that $\theta_2(F_{xy})\geq 1$. Then there is an edge $uv\in E(F_{xy}-x-y)$ with $u\in U$ and $v\in V$ such that $\deg_{F_{xy}}(u,T)=\deg_{F_{xy}}(v,T)=1$, say $N_{F_{xy}}(u,T)=\{u^*\}$ and $N_{F_{xy}}(v,T)=\{v^*\}$. Let $C$ be the cycle in $F_{xy}$ containing $u,v,u^*,v^*,x,y$.
Assume that $C=xy\ldots u^*uvv^*\ldots x$ or $C=yx\ldots u^*uvv^*\ldots y$. We now distinguish the following two cases.

{\bf Case 1. } $C=xy\ldots u^*uvv^*\ldots x$.

If $y\ldots u^*u$ has odd length $l$, then by Lemma \ref{lem-3-4} (i) with $W=V(F_{xy})$ there is a $yu$-path $P$ of length $l$ in $G[U,V]$ such that $V(P)\cap V(F_{xy})=\{y,u\}$. By replacing $y\ldots u^*u$ from $F_{xy}$ with  $P$, we obtain a new copy $F_{xy}'$ of $F$ with $F_{xy}'\in \hf_{xy}$ and $\theta_2(F_{xy}')<\theta_2(F_{xy})$, contradicting the choice of $F_{xy}$. If $y\ldots u^*u$ has even length $l$, then  the path $vv^*\ldots x$ has odd length $2k-1-l$. By Lemma \ref{lem-3-4} (i) with $W=V(F_{xy})$ there is a $vx$-path $Q$ of length $2k-1-l$ in $G[U,V]$ such that $V(Q)\cap V(F_{xy})=\{v,x\}$. By replacing $vv^*\ldots x$ from $F_{xy}$ with  $Q$, we obtain a new copy $F_{xy}'$ of $F$ with $F_{xy}'\in \hf_{xy}$ and $\theta_2(F_{xy}')<\theta_2(F_{xy})$, contradicting the choice of $F_{xy}$.

{\bf Case 2. } $C=yx\ldots u^*uvv^*\ldots y$.

If $x\ldots u^*u$ has even length $l$, then by Lemma \ref{lem-3-4} (ii) with $W=V(F_{xy})$ there is a $xu$-path $P$ of length $l$ in $G[U,V]$ such that $V(P)\cap V(F_{xy})=\{x,u\}$. By replacing $x\ldots u^*u$ from $F_{xy}$ with  $P$, we obtain a new copy $F_{xy}'$ of $F$ with $F_{xy}'\in \hf_{xy}$ and $\theta_2(F_{xy}')<\theta_2(F_{xy})$, contradicting the choice of $F_{xy}$. If $x\ldots u^*u$ has odd length $l$, then  the path $vv^*\ldots y$ has even length $2k-1-l$. By Lemma \ref{lem-3-4} (ii) with $W=V(F_{xy})$ there is a $vy$-path $Q$ of length $2k-1-l$ in $G[U,V]$ such that $V(Q)\cap V(F_{xy})=\{v,y\}$. By replacing $vv^*\ldots y$ from $F_{xy}$ with  $Q$, we obtain a new copy $F_{xy}'$ of $F$ with $F_{xy}'\in \hf_{xy}$ and $\theta_2(F_{xy}')<\theta_2(F_{xy})$, contradicting the choice of $F_{xy}$.

Thus $\theta_1(F_{xy})=\theta_2(F_{xy})=0$ and the claim follows.
\end{proof}
By Claim 4, both $N_{F_{xy}}(x)\cap T$ and  $N_{F_{xy}}(y)\cap T$ are not empty, and let
$N_{F_{xy}}(x)\cap T =\{x_1^*,x_2^*,\ldots,x_p^*\}$, $N_{F_{xy}}(y)\cap T=\{y_1^*,y_2^*,\ldots,y_q^*\}$.
Then $\{x_1^*,x_2^*,\ldots,x_p^*\}\cap \{y_1^*,y_2^*,\ldots,y_q^*\} =\emptyset$ since $F$ is $K_3$-free.
Let $N_{F_{xy}}(x)\cap V=\{z_1,z_2,\ldots,z_f,z_{f+1},\ldots,z_{s-p}\}$ such that $N_{F_{xy}}(z_\ell)\cap T\neq \emptyset$ for $\ell\leq f$ and $N_{F_{xy}}(z_\ell)\cap T= \emptyset$ for $f+1\leq \ell\leq s-p$. In $F_{xy}$, each $z_\ell$ ($\ell=1,\ldots,f$) has one neighbor being $x$ and the other one $z_\ell^*$ in $T$, and each $z_\ell$ ($\ell=f+1,\ldots,s-p$) has one neighbor being $x$ and the other one $u_\ell$ in $U$. Let $X$ be the set of common neighbors of $x_1^*,x_2^*,\ldots,x_p^*$ in $U_1$ and let $Z_\ell$ be the set of neighbors of $z_\ell^*$ in $V$ for each $\ell=1,\ldots, f$. Similarly, $N_{F_{xy}}(y)\cap U=\{w_1,w_2,\ldots,w_g,w_{g+1},\ldots,w_{s-q}\}$ such that $N_{F_{xy}}(w_\ell)\cap T\neq \emptyset$ for $\ell\leq g$ and $N_{F_{xy}}(w_\ell)\cap T= \emptyset$ for $g+1\leq \ell\leq s-q$. In $F_{xy}$, each $w_\ell$ ($\ell=1,\ldots,g$) has one neighbor being $y$ and the other one $w_\ell^*$ in $T$, and each $w_\ell$ ($\ell=g+1,\ldots,s-q$) has one neighbor being $y$ and the other one $v_\ell$ in $V$. Let $Y$ be the set of common neighbors of $y_1^*,y_2^*,\ldots,y_q^*$ in $V_1$ and let $W_\ell$ be the set of neighbors of $w_\ell^*$ in $U$ for each $\ell=1,\ldots, g$ as shown in Figure \ref{structureofFxy}.

\begin{figure}
  \centering
  \includegraphics[width=0.8\textwidth]{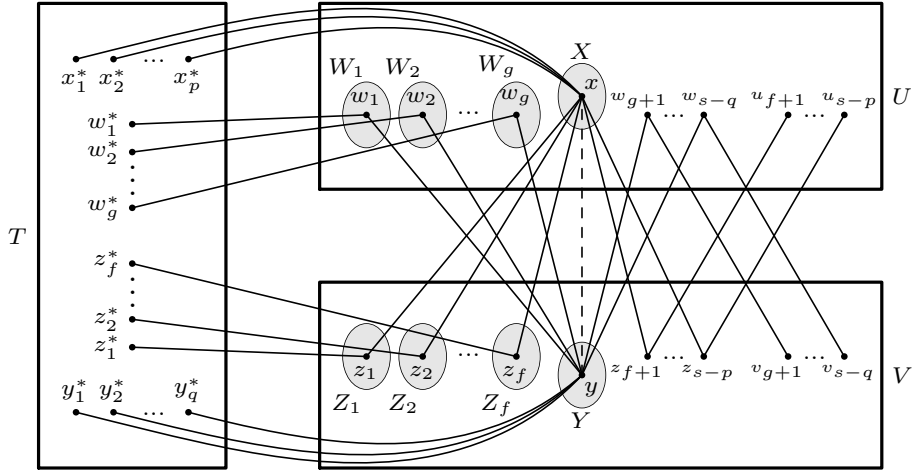}
  \caption{The local structure of $F_{xy}$ with $xy\in \Omega_2$.}\label{structureofFxy}
\end{figure}

For distinct vertices $x'\in X, z_1'\in Z_1,\ldots, z_f'\in Z_f$, if $x'z_1',x'z_2',\ldots,x'z_f'$ are  edges in $G$ then we say that $G[x';z_1',\ldots,z_f']$ is an {\it $(X,Z_1,\ldots,Z_f)$-star} with center $x'$. Let
\[
X_0=\left\{x'\in X\colon \mbox{there exists } \mbox{an } (X,Z_1,\ldots,Z_f)\mbox{-star with center }x'  \right\}.
\]
Clearly, $G[x;z_1,\ldots,z_f]$ is an $(X,Z_1,\ldots,Z_f)$-star, implying that $x\in X_0$.
Similarly, let
\[
Y_0=\left\{y'\in Y\colon \mbox{there exists } \mbox{a } (Y,W_1,\ldots,W_g)\mbox{-star with center }y' \right\}
\]
and clearly $y\in Y_0$.

For any pair $(x',y')$ with $x'\in X_0$ and $y'\in Y_0$, if there exist  $z_1',\ldots,z_f'$ such that $G[x';z_1',\ldots,z_f']$ is an $(X,Z_1,\ldots,Z_f)$-star and  $y'\notin \{z_1',\ldots,z_f'\}$, then we say that $y'$ is {\it good} to $x'$; otherwise $y'$ is {\it bad} to $x'$. Similarly, if there exist $w_1',\ldots,w_g'$ such that $G[y';w_1',\ldots,w_g']$ is a $(Y,W_1,\ldots,W_g)$-star and  $x'\notin \{w_1',\ldots,w_g'\}$, then we say that $x'$ is good to $y'$, otherwise $x'$ is bad to $y'$.  We call $(x',y')$ a {\it compatible} pair if $y'$ is good to $x'$ and $x'$ is good to $y'$; otherwise we say that $(x',y')$ is {\it incompatible}. For each $x'\in X_0$, there exist $z_1',\ldots,z_f'$ such that $G[x';z_1',\ldots,z_f']$ is an $(X,Z_1,\ldots,Z_f)$-star, implying that the number of vertices in $Y_0$ that are bad to $x'$ is at most $f$. Similarly, for each $y'\in Y_0$  the number of vertices in $X_0$ that are bad to $y'$ is at most $g$. Then the number of incompatible pairs between $X_0$ and $Y_0$ is at most $f|X_0|+g|Y_0|$. Thus, the number of compatible pairs between $X_0$ and $Y_0$ is at least $|X_0||Y_0|-f|X_0|-g|Y_0|$.

\begin{claim}
Every compatible pair $(x',y')$ with $x'\in X_0$ and $y'\in Y_0$ is a non-edge of $G[U_1,V_1]$.
\end{claim}
\begin{proof}
Suppose not, let $(x',y')$ with $x'\in X_0$, $y' \in Y_0$ be a compatible pair  and $x'y'\in E(G)$. Then there exist $z_1',\ldots,z_f'$ and $w_1',\ldots,w_g'$ such that $G[x',z_1',\ldots,z_f']$ is an $(X,Z_1,\ldots,Z_f)$-star and $G[y';w_1',\ldots,w_g']$ is a $(Y,W_1,\ldots,W_g)$-star. We shall find a copy of $F$ in $G$, which leads to a contradiction.

Let
\[
R_1'=\{x',y',z_1',\ldots,z_f',w_1',\ldots,w_g'\}.
\]
Since $u_{\ell}$ $(f+1\leq \ell\leq s-p)$ and $x'$ have at least $\left(\frac{1}{2}-\frac{1}{10h}\right)n$ common neighbors in $V$ and $n$ is sufficiently large, we may choose distinct $z_{f+1}', \ldots,z_{s-p}'$ from $V\setminus (V(F_{xy})\cup R_1')$ such that $z_{\ell}'\in N(x',V)\cap N(u_\ell,V)$ for each $\ell=f+1,\ldots,s-p$. Similarly, we may choose distinct $w_{g+1}', \ldots,w_{s-q}'$ from $U\setminus (V(F_{xy})\cup R_1')$ such that $w_{\ell}'\in N(y',U)\cap N(v_\ell,U)$ for each $\ell=g+1,\ldots,s-q$. Let
\[
R_1=\{x,y,z_1,\ldots,z_f,w_1,\ldots,w_g\},\ R_2=\{z_{f+1},\ldots,z_{s-p},w_{g+1},\ldots,w_{s-q}\}
\]
 and
\[
 R_2'=\{z_{f+1}',\ldots,z_{s-p}',w_{g+1}',\ldots,w_{s-q}'\}.
\]
Clearly, $R_2'\cap (R_1'\cup V(F_{xy}))=\emptyset$. Let $F^0$ be a graph obtained from $F_{xy}$ by replacing vertices in $R_1\cup R_2$ with vertices in $R_1'\cup R_2'$. If $R_1'\cap (V(F_{xy})\setminus (R_1\cup R_2))=\emptyset$, then $F^0$ is a copy of $F$ in $G$, a contradiction.
Hence  $R_1'\cap (V(F_{xy})\setminus (R_1\cup R_2))\neq \emptyset$, that is, $F^0$ is the image of an $F$-homomorphism but not
a copy of $F$.

Now we replace the overlapped vertices in $F^0$ to get a copy of $F$ by a greedy algorithm. Let $\phi_0$ be the homomorphism from $F$ to $F_0$ and let $\phi_{xy}$ be the isomorphism from $F$ to $F_{xy}$. Then $\phi_0$ is not an isomorphism and $\phi_{xy}$ is an isomorphism. Let $R_1^*=\phi_{xy}^{-1}(R_1)$ and  $R_2^*=\phi_{xy}^{-1}(R_2)$. By the labeling of $F$, we see that
\begin{align*}
R_1^*\cup R_2^*= \left\{a,b\right\}\cup  \left\{c_1^i  \colon \phi_{xy}(c_1^i)\in U\cup V\right\}\cup \left\{c_{2k-1}^i  \colon \phi_{xy}(c_{2k-1}^i)\in U\cup V\right\}.
\end{align*}
Let
\begin{align}\label{eq-R3star}
 R_3^*= \left\{u\in V(F) \colon \phi_{xy}(u)\in T\right\} \mbox{ and } R_4^*=\left(\cup_{i=1}^s\left\{c_2^i,\ldots,c_{2k-2}^i\right\}\right)\setminus R_3^*.
\end{align}
Clearly, $(R_1^*,R_2^*,R_3^*,R_4^*)$ is a partition of $V(F)$ and $\phi_{xy}(R_4^*)\subset U\cup V$. Since for any edge $uv$ of $F$ with $u\in R_1^*$ we have $\phi_{xy}(v)\in T\cup R_1\cup R_2$, it follows that $v\in R_1^*\cup R_2^*\cup R_3^*$. Thus there is no edge between $R_1^*$ and $R_4^*$ in $F$.
Note that $\phi_0$ can be expressed explicitly as follows:
\begin{itemize}
    \item[(i)] $\phi_0(\phi_{xy}^{-1}(x))=x'$ and $\phi_0(\phi_{xy}^{-1}(y))=y'$;
    \item[(ii)] for each $\ell =1,\ldots,s-p$, $\phi_0(\phi_{xy}^{-1}(z_\ell))=z_{\ell}'$  and  for  each $\ell =1,\ldots,s-q$, $\phi_0(\phi_{xy}^{-1}(w_\ell))=w_{\ell}'$;
    \item[(iii)] $\phi_0(c)=\phi_{xy}(c)$ for  $c\in R_3^*\cup R_4^*$.
\end{itemize}
Since vertices in $R_2'$ are chosen disjoint from $V(F_{xy})\cup R_1'$, we have $\phi_0(R_2^*)\cap \phi_0(R_1^*\cup R_3^*\cup R_4^*)=\emptyset$. Then $\phi_0(R_1^*)\cap \phi_0(R_4^*)\neq \emptyset$ because $\phi_0$ is not an isomorphism and $\phi_{xy}$ is an isomorphism.

For any $a\in R_1^*$ and $c\in R_4^*$ with $\phi_0(a)=\phi_0(c)$, $\phi_0(c)=\phi_0(a)\in \phi_0(R_1^*)=R_1'\subset U\cup V$. 
By \eqref{eq-R3star} we have $c\in \cup_{i=1}^s\{c_2^i,\ldots,c_{2k-2}^i\}$. Let $d_1,d_2$ be two neighbors of $c$ in $F$. Clearly, $d_i$ has degree two in $F$ for $i=1,2$. Since $c\in R_4^*$ and there is no edge between $R_1^*$ and $R_4^*$ in $F$, it follows that $d_1,d_2\notin R_1^*$.
Since $\phi_0(c)=\phi_{xy}(c)\in V(F_{xy})\setminus (T\cap \{x,y\})$, at most one of $\phi_{xy}(d_1),\phi_{xy}(d_2)$ is in $T$ by Claim 5 (i). Recall that $F^0$ is obtained from $F_{xy}$ by replacing vertices in $R_1\cup R_2$ with vertices in $R_1'\cup R_2'$ and  never changing vertices in $T$.  Thus $|\{\phi_0(d_1),\phi_0(d_2)\}\cap T|= |\{\phi_{xy}(d_1),\phi_{xy}(d_2)\}\cap T|\leq 1$. We shall find an $F$-homomorphism $\phi_1$ such that $|\phi_1(V(F))|>|\phi_0(V(F))|$ by distinguishing two cases.

{\bf Case 1.} $|\{\phi_0(d_1),\phi_0(d_2)\}\cap T|=0$.

Without loss of generality, we assume that $\phi_0(c)\in U$ and $\phi_0(d_1),\phi_0(d_2)\in V$. Since  $\phi_0(d_1),\phi_0(d_2)$ have at least $\left(\frac{1}{2}-\frac{1}{10h}\right)n$ common neighbors in $U$, we may choose $u'$ from $N(\phi_0(d_1))\cap N(\phi_0(d_2))\setminus \phi_0(V(F))$. Define $\phi_1(c)=u'$ and $\phi_1(a)=\phi_0(a)$ for all $a\in V(F)\setminus \{c\}$. It is easy to see that $\phi_1$ is  an $F$-homomorphism with $|\phi_1(V(F))|=|\phi_0(V(F))|+1$.

{\bf Case 2.} $|\{\phi_0(d_1),\phi_0(d_2)\}\cap T|=1$.

Without loss of generality, we assume that $\phi_0(c)\in U$, $\phi_0(d_1)\in V$ and $\phi_0(d_2)\in T$.
Recall that $d_1$ has exactly two neighbors in $F$ and one of them is $c$, and let $d_3$ be the other one. Since $\phi_{xy}(cd_1)$ is an edge of $F_{xy}-T-\{x,y\}$, by Claim 5 (ii) $\deg_{F_{xy}}(\phi_{xy}(c),T)+\deg_{F_{xy}}(\phi_{xy}(d_1),T)\leq 1$, that is,
$|\phi_{xy} (\{d_1,d_2\})\cap T|+|\phi_{xy} (\{c,d_3\})\cap T|\leq 1$. Because $F^0$ is obtained from $F_{xy}$ by replacing vertices in $R_1\cup R_2$ with vertices in $R_1'\cup R_2'$ and  never changing vertices in $T$, we have $|\phi_{0} (\{d_1,d_2\})\cap T|+|\phi_{0} (\{c,d_3\})\cap T|=|\phi_{xy} (\{d_1,d_2\})\cap T|+|\phi_{xy} (\{c,d_3\})\cap T|\leq 1$. Then $|\phi_{0} (\{c,d_3\})\cap T|=0$ by $|\phi_{0} (\{d_1,d_2\})\cap T|=1$, implying that $\phi_0(d_3)\in U$.
Since $\phi_0(d_2)$ has one neighbor  $\phi_0(c)$ in $U$, by Lemma \ref{lem-3-3} (iii) we know that $\phi_0(d_2)$ has at least $h+1$ neighbors in $U$, and let $u'\in N(\phi_0(d_2),U)\setminus \phi_0(V(F))$. Moreover, since $u'$ and $\phi_0(d_3)$ have at least $\left(\frac{1}{2}-\frac{1}{10h}\right)n>h$ common neighbors in $V$,  we may choose $v'\in N(u',V)\cap N(\phi_0(d_3),V)\setminus \phi_0(V(F))$.  Define $\phi_1(c)=u'$, $\phi_1(d_1)=v'$ and $\phi_1(a)=\phi_0(a)$ for all $a\in V(F)\setminus \{c,d_1\}$. It is easy to see that $\phi_1$ is  an $F$-homomorphism with $|\phi_1(V(F))|\geq |\phi_0(V(F))|+1$.

If $\phi_1$ is not an $F$-isomorphism, then there exist $a'\in R_1^*$ and $c'\in R_4^*$ with $\phi_1(a')=\phi_1(c')$. By the same argument above, we shall find an $F$-homomorphism $\phi_2$ such that $|\phi_2(V(F))|>|\phi_1(V(F))|$.  Do this repeatedly, we get   $F$-homomorphisms $\phi_1,\phi_2,\ldots, \phi_l,\ldots$ with $h-|R_1'|\leq |\phi_0(V(F))|< |\phi_1(V(F))|<\cdots<|\phi_l(V(F))|<\cdots$.
Since $|\phi_i(V(F))|\leq h$ for all $i$, we shall obtain an $F$-isomorphism in at most $|R_1'|$ steps, contradicting the fact that $G$ is $F$-free. Thus, every compatible pair $(x',y')$ with $x'\in X_0$ and $y'\in Y_0$ is not an edge in $G[U_1,V_1]$.
\end{proof}

Recall that the number of compatible pairs between $X_0$ and $Y_0$ is at least $|X_0||Y_0|-f|X_0|-g|Y_0|$. Since $f,g\leq h$, it follows that
\[
e_{\bar{G}}(X_0,Y_0) \geq |X_0||Y_0|-f|X_0|-g|Y_0| \geq |X_0||Y_0|-h(|X_0|+|Y_0|).
\]
Let $S$ be one of $X_0$ and $Y_0$ with smaller size.  By the same argument as in the proof of Step 1, we have
\[
e_{\bar{G}}(X_0,Y_0) \geq \frac{|S|^2}{16h^2}.
\]
We delete vertices in $S$ from $U_1\cup V_1$ and let $U_1'=U_1\setminus S$ and $V_1'=V_1\setminus S$. If $\bar{E}[U_1',V_1'] \cap \Omega_2=\emptyset$, then we are done. Otherwise, there is another non-edge $xy$ in $\Omega_2$ with $x\in U_1'$, $y\in V_1'$, and we delete another $S'$ from $U_1'\cup V_1'$  incidents with at least $\frac{|S'|^2}{16h^2}$ non-edges between $U_1'$ and $V_1'$. By deleting vertices greedily, we shall obtain a sequence of disjoint sets $S_1,S_2,\ldots,S_l$ in $U_1\cup V_1$ such that $\bar{E}[U_1\setminus (S_1\cup \ldots \cup S_l),V_1\setminus (S_1\cup \ldots \cup S_l)] \cap \Omega_2=\emptyset$.

In each step of the greedy algorithm, there are vertices $x_1^*,\ldots,x_p^*, z_1^*,\ldots z_f^*\in T$ and $y_1^*,\ldots,y_q^*, w_1^*,\ldots w_g^*\in T$ such that either $X_0$ or $Y_0$ is deleted. If $X_0$ is deleted, then since $X$ is the set of common neighbors of $x_1^*,\ldots,x_p^*$ in the  $U_1\setminus X_0$, there are no $(X,Z_1,\ldots,Z_f)$-stars in the future steps. It follows that the tuple $(x_1^*,\ldots,x_p^*, z_1^*,\ldots z_f^*)$ will not appear in the future steps of the algorithm. Similarly, if $Y_0$ is deleted, then the tuple $(y_1^*,\ldots,y_q^*, w_1^*,\ldots w_g^*)$ will not appear in the future steps of the algorithm. Since $p+f\leq s$ and $q+g\leq s$, it follows that
\begin{align*}
l&\leq \sum_{p+f\leq s} \binom{|T|}{p} \binom{|T|-p}{f}  +\sum_{q+g\leq s} \binom{|T|}{q} \binom{|T|-q}{g}\\
&= 2\sum_{p+f\leq s}\binom{|T|}{p} \binom{|T|-p}{f}\\
&\leq 2 \sum_{p+f\leq s} |T|^s\leq 2 s^2 |T|^s.
\end{align*}

Similarly, by \eqref{ineq3-3} and \eqref{ineq3-4} we arrive at
\[
\left(\sum_{i=1}^{l} |S_i|\right)^2  \leq 16h^2 \cdot 25h^2\varepsilon n^2 l\leq  20^2 h^4\varepsilon n^2 \cdot 2s^2|T|^s \leq  20^2 h^4\varepsilon n^2 \cdot 2s^2 (30h^2\varepsilon n)^s.
\]
Let $U_2=U_1\setminus (S_1\cup \ldots \cup S_l)$ and $V_2=V_1\setminus (S_1\cup \ldots \cup S_l)$. Then
\[
|U_1\setminus U_2|+|V_1\setminus V_2| \leq \sum_{i=1}^{l} |S_i| \leq 20\sqrt{2}\cdot 30^{\frac{s}{2}}h^{s+2}s\varepsilon^{\frac{s+1}{2}} n^{\frac{s+2}{2}}\leq (6h)^{s+2}\varepsilon^{\frac{s+1}{2}} n^{\frac{s+2}{2}}
\]
and Step 2 is finished.
\end{proof}

By Step 2, we see that $\bar{E}[U_2,V_2] \cap (\Omega_1\cup \Omega_2)=\emptyset$. Thus, we are left to delete vertices from $U_2\cup V_2$ to destroy all non-edges in $\Omega_3\cap \bar{E}[U_2,V_2]$. If $\Omega_3\cap \bar{E}[U_2,V_2]=\emptyset$, we have nothing to do. Hence we assume that $\Omega_3\cap \bar{E}[U_2,V_2]\neq\emptyset$ and  let $xy$ be a non-edge in $\Omega_3$ with $x\in U_2$ and $y\in V_2$. By definition of $\Omega_3$, there is at least one copy $F_{xy}$ of $F$ such that $\deg_{F_{xy}}(x)=s+1$ and $\deg_{F_{xy}}(y)=2$ or $\deg_{F_{xy}}(x)=2$ and $\deg_{F_{xy}}(y)=s+1$. We  partition $\Omega_3$ into four classes as follows:
\begin{equation*}
\left\{
\begin{array}{l}\Omega_{31}=\left\{xy\in \Omega_3\colon \deg_{F_{xy}}(x)=s+1,\ \deg_{F_{xy}}(y)=2,\ \deg_{F_{xy}}(x,T)=s \right\},\\[6pt]
\Omega_{32}=\left\{xy\in \Omega_3\colon \deg_{F_{xy}}(x)=2,\ \deg_{F_{xy}}(y)=s+1,\ \deg_{F_{xy}}(y,T)=s \right\},\\[6pt]
\Omega_{33}=\left\{xy\in \Omega_3\colon  \deg_{F_{xy}}(x)=s+1,\ \deg_{F_{xy}}(y)=2,\ \deg_{F_{xy}}(x,T)\leq s-1\right\},\\[6pt]
\Omega_{34}=\left\{xy\in \Omega_3\colon \deg_{F_{xy}}(x)=2,\ \deg_{F_{xy}}(y)=s+1,\ \deg_{F_{xy}}(y,T)\leq s-1 \right\}.
\end{array}\right.
\end{equation*}

We complete the proof by the following two steps.

{\noindent \bf Step 3.1.}
We can find $U_3\subset U_2$ and $V_3\subset V_2$ such that $|U_2\setminus U_3|+|V_2\setminus V_3|  \leq (6h)^{s+3}\varepsilon^{\frac{s+1}{2}} n^{\frac{s+2}{2}}$ and $\bar{E}[U_3,V_3]\cap(\Omega_{31}\cup \Omega_{32}) =\emptyset$.  That is,  by deleting at most $(6h)^{s+3}\varepsilon^{\frac{s+1}{2}} n^{\frac{s+2}{2}}$ vertices from $U_2\cup V_2$ we destroy all non-edges in $\Omega_{31}\cup \Omega_{32}$.

\begin{proof}
If $\bar{E}[U_2,V_2]\cap(\Omega_{31}\cup \Omega_{32}) = \emptyset$, we have nothing to do. So assume that $\bar{E}[U_2,V_2]\cap(\Omega_{31}\cup \Omega_{32}) \neq \emptyset$, then there is a non-edge $xy$ in $\Omega_{31}\cup \Omega_{32}$ with $x\in U_2$ and $y\in V_2$. Without loss of generality, we assume that $xy\in \Omega_{31}$. By Claim 4, $\deg_{F_{xy}}(y,T)=1$ and let $ N_{F_{xy}}(y,T)=\{y^*\}$. Assume that
\[
N_{F_{xy}}(x)\cap T =\{x_1^*,x_2^*,\ldots,x_s^*\}.
\]
Let $X$ be the set of common neighbors of $x_1^*,\ldots,x_s^*$ in $U_2$ of $G$ and let $Y=N_G(y^*,V_2)$. For any edge $x'y'$ in $G[X,Y]$, if $\{x',y'\}\cap V(F_{xy})=\emptyset$, then by replacing $x,y$ with $x',y'$ in $F_{xy}$ we obtain a copy of $F$ in $G$, a contradiction. Thus,
every edge in $G[X,Y]$ intersects $V(F_{xy})$, implying that $e(X,Y)\leq h(|X|+|Y|)$. Then
\[
e_{\bar{G}}(X,Y) = |X||Y| - e(X,Y) \geq |X||Y|-h(|X|+|Y|).
\]
Let $S$ be one of $X$ and $Y$ with the smaller size. By the same argument as in Step 1, we have
\[
e_{\bar{G}}(X,Y) \geq \frac{|S|^2}{16h^2}.
\]
We delete vertices in $S$ from $U_2\cup V_2$ and let $U_2'=U_2\setminus S$ and $V_2'=V_2\setminus S$.
If $\bar{E}[U_2',V_2'] \cap(\Omega_{31}\cup \Omega_{32}) =\emptyset$, then we are done. Otherwise, there is another non-edge $xy$ in $\Omega_{31}\cup \Omega_{32}$ with $x\in U_2'$, $y\in V_2'$, and we delete another $S'$ from $U_2'\cup V_2'$ incidents with at least $\frac{|S'|^2}{16h^2}$ non-edges between $U_2'$ and $V_2'$. By deleting vertices greedily, we shall obtain a sequence of disjoint sets $S_1,S_2,\ldots,S_l$ in $U_2\cup V_2$ such that $\bar{E}[U_2\setminus (S_1\cup \ldots \cup S_l),V_2\setminus (S_1\cup \ldots \cup S_l)] \cap (\Omega_{31}\cup \Omega_{32})=\emptyset$.

In each step of the greedy algorithm, if there is a non-edge $xy\in \Omega_{31}$ between $U_2$ and $V_2$, then there exist vertices $x_1^*,\ldots,x_s^*, y^*\in T$ such that either $X=\cap_{i=1}^s N(x_i^*,U_2)$ or $Y=N(y^*,V_2)$ is deleted. If there is  a non-edge $xy\in \Omega_{32}$ between $U_2$ and $V_2$, then there exist vertices $y_1^*,\ldots,y_s^*, x^*\in T$ such that either $X=N(x^*,U_2)$ or $Y=\cap_{i=1}^s N(y_i^*,V_2)$ is deleted. It follows that
\[
l\leq 2\left(\binom{|T|}{s}+|T|\right)< 2(|T|^s+|T|)\leq 4|T|^s\leq 4 (30h^2\varepsilon n)^s.
\]
By \eqref{ineq3-3} and \eqref{ineq3-4}, we arrive at
\[
\left(\sum_{i=1}^{l} |S_i|\right)^2  \leq 16h^2 \cdot 25h^2\varepsilon n^2 l \leq  20^2 h^4\varepsilon n^2 \cdot 4 (30h^2\varepsilon n)^s=40^2\cdot30^sh^{2s+4}\varepsilon^{s+1}n^{s+2}.
\]
Let $U_3=U_2\setminus (S_1\cup \ldots \cup S_l)$ and $V_3=V_2\setminus (S_1\cup \ldots \cup S_l)$. Then
\[
|U_2\setminus U_3|+|V_2\setminus V_3| \leq \sum_{i=1}^{l} |S_i| \leq 40\cdot 30^{\frac{s}{2}}h^{s+2}\varepsilon^{\frac{s+1}{2}} n^{\frac{s+2}{2}}\leq (6h)^{s+3}\varepsilon^{\frac{s+1}{2}} n^{\frac{s+2}{2}}
\]
and Step 3.1 is finished.
\end{proof}

{\noindent \bf Step 3.2.}
We can find $U_4\subset U_3$ and $V_4\subset V_3$ such that $|U_3\setminus U_4|+|V_3\setminus V_4|  \leq (6h)^{s+3}\varepsilon^{\frac{s+1}{2}} n^{\frac{s+2}{2}}$ and $G[U_4,V_4]$ is complete bipartite, i.e.,  by deleting at most $(6h)^{s+3}\varepsilon^{\frac{s+1}{2}} n^{\frac{s+2}{2}}$ vertices from $U_3\cup V_3$ we obtain  an induced complete bipartite subgraph of $G$.

\begin{proof}
If $\bar{E}[U_3,V_3]\cap(\Omega_{33}\cup \Omega_{34}) = \emptyset$, we have nothing to do. So assume that $\bar{E}[U_3,V_3]\cap(\Omega_{33}\cup \Omega_{34})\neq \emptyset$, then there is a non-edge $xy$ in $\Omega_{33}\cup \Omega_{34}$ with $x\in U_3$ and $y\in V_3$. Without loss of generality, we assume that $\deg_{F_{xy}}(x)=s+1$ and $\deg_{F_{xy}}(y)=2$. By Claim 4, $\deg_{F_{xy}}(y,T)=1$ and let $N_{F_{xy}}(y,T)=\{y^*\}$. Let $N_{F_{xy}}(x)\cap T =\{x_1^*,x_2^*,\ldots,x_p^*\}$ and $N_{F_{xy}}(x)\cap V =\{z_{p+1},\ldots,z_s, y\}$ with $p\leq s-1$. By Claim 4, we have $p\geq 1$. Let $X$ be the set of common neighbors of $x_1^*,\ldots,x_p^*$ in $U_3$ of $G$ and let $Y=N_G(y^*,V_3)$.

Let $\phi_{xy}$ be the isomorphism from $F$ to $F_{xy}$. Without loss of generality, assume that $\phi_{xy}(a)=x$ and $\phi_{xy}(c_1^1)=y$.  If $\psi$ is an injective homomorphism from $F-c_1^1$ to $G$ with
\[
\psi(\phi_{xy}^{-1}(x_1^*))=x_1^*,\ \ldots,\ \psi(\phi_{xy}^{-1}(x_p^*))=x_p^*,\ \psi(\phi_{xy}^{-1}(y^*))=y^*
\]
and
\[
\psi(\phi_{xy}^{-1}(z_{p+1}))\in V,\ \ldots,\ \psi(\phi_{xy}^{-1}(z_s))\in V,\  \psi(a) \in X,
\]
then we say that $\psi$ is {\it agree with} $\phi_{xy}$. Define
\[
\Psi_{xy} =\{\psi\colon \psi \mbox{ is an injective homomorphism from $F-c_1^1$ to $G$ that is agree with }  \phi_{xy} \}.
\]
Let
\[
X_0=\{x'\in X\colon \exists \psi_{x'}\in \Psi_{xy} \mbox{ such that } \psi_{x'}(a)=x'\}.
\]
For any $x'\in X_0$, let $H_{x'}$ be a copy of $F-c_1^1$ in $G$ corresponding to $\psi_{x'}$. If there is $y'\in Y\setminus V(H_{x'})$ such that $x'y'\in E(G)$, then we define a homomorphism $\phi$ from $F$ to $G$ as follows:
\[
\phi(u)=\psi_{x'}(u) \mbox{ for all }u \in V(F-c_1^1)\mbox{ and }\phi(c_1^1)= y'.
\]
Then $\phi$ is an injective homomorphism from $F$ to $G$, contradicting the fact that $G$ is $F$-free. Hence each $x'\in X_0$ has at most $V(H_{x'})$ neighbors in $Y$, implying that
\[
e_{\bar{G}}(X_0,Y)\geq |X_0||Y|-|X_0|h\geq |X_0||Y|-h(|X_0|+|Y|).
\]
Let $S$ be one of $X_0$ and $Y$ with the smaller size. By the same argument as in Step 1, we have
\[
e_{\bar{G}}(X_0,Y) \geq \frac{|S|^2}{16h^2}.
\]

We delete vertices in $S$ from $U_3\cup V_3$ and let $U_3'=U_3\setminus S$ and $V_3'=V_3\setminus S$.
If $\bar{E}[U_3',V_3'] \cap (\Omega_{33}\cup \Omega_{34})=\emptyset$, then we are done. Otherwise, there is another non-edge $xy$ in $(\Omega_{33}\cup \Omega_{34})$ with $x\in U_3'$, $y\in V_3'$, and we delete another $S'$ from $U_3'\cup V_3'$ incident with at least $\frac{|S'|^2}{16h^2}$ non-edges between $U_3'$ and $V_3'$. By deleting vertices greedily, we shall obtain a sequence of disjoint  sets $S_1,S_2,\ldots,S_l$ in $U_3\cup V_3$ such that $\bar{E}[U_3\setminus (S_1\cup \ldots \cup S_l),V_3\setminus (S_1\cup \ldots \cup S_l)] \cap (\Omega_{33}\cup \Omega_{34})=\emptyset$.

In each step of the greedy algorithm, if there is a non-edges $xy\in \Omega_{33}$ between $U_3$ and $V_3$, then there exist vertices $x_1^*,\ldots,x_p^*, y^*\in T$ such that either $X_0$ or $Y$ is deleted. If $Y$ is deleted, then $y^*$ has no neighbor in  $V_3\setminus Y$. If $X_0$ is deleted, then there is no non-edge $x'y'\in \Omega_{33}$ between $U_3\setminus X_0$ and $V_3$ such that $N_{F_{x'y'}}(x')\cap T =\{x_1^*,x_2^*,\ldots,x_p^*\}$ and $N_{F_{x'y'}}(y')\cap T =\{y^*\}$. For otherwise, $F_{x'y'}-y'$ is a copy of $F-\phi_{x'y'}^{-1}(y')$, which is also a copy of $F-c_1^1$,  contradicting the assumption that $x'\notin U_3\setminus X_0$. It follows that the tuple $(x_1^*,\ldots,x_p^*, y^*)$ will not appear in the future steps of the algorithm. If there is a non-edges $xy\in \Omega_{34}$ between $U_3$ and $V_3$, then there exist vertices $y_1^*,\ldots,y_q^*, x^*\in T$ such that either $X=N(x^*,U_2)$ or $Y_0$ (which can be defined similarly) is deleted. By the same argument, we see that the tuple $(y_1^*,\ldots,y_q^*, x^*)$ will not appear in the future steps of the algorithm. Therefore,
\[
l\leq \sum_{p=1}^{s-1} \binom{|T|}{p}|T| +\sum_{q=1}^{s-1} \binom{|T|}{q}|T|< \sum_{p=1}^{s-1} |T|^{p+1} +\sum_{q=1}^{s-1} |T|^{q+1}< 2s|T|^s\leq 2s (30h^2\varepsilon n)^s.
\]
By \eqref{ineq3-3} and \eqref{ineq3-4}, we arrive at
\[
\left(\sum_{i=1}^{l} |S_i|\right)^2  \leq 16h^2 \cdot 25h^2\varepsilon n^2 l \leq  20^2 h^4\varepsilon n^2 \cdot 2s (30h^2\varepsilon n)^s = 2\cdot 20^2 \cdot 30^s sh^{2s+4}\varepsilon^{s+1}n^{s+2}.
\]
Let $U_4=U_3\setminus (S_1\cup \ldots \cup S_l)$ and $V_4=V_3\setminus (S_1\cup \ldots \cup S_l)$. Since $\bar{E}[U_4,V_4] \cap (\Omega_1\cup\Omega_2\cup\Omega_3)=\emptyset$, $G[U_4,V_4]$ is complete bipartite. Moreover,
\[
|U_3\setminus U_4|+|V_3\setminus V_4| \leq \sum_{i=1}^{l} |S_i| \leq 20\sqrt{2}\cdot 30^{\frac{s}{2}}\sqrt{s}h^{s+2}\varepsilon^{\frac{s+1}{2}} n^{\frac{s+2}{2}}\leq (6h)^{s+3}\varepsilon^{\frac{s+1}{2}} n^{\frac{s+2}{2}}.
\]
Thus Step 3.2 is finished.
\end{proof}

Let $n'$ be the total number of vertices we deleted from $G$ to obtain an induced complete bipartite graph.  By Lemma \ref{lem-3-3} (i) and Steps 1, 2, 3.1, 3.2, we have
\begin{align*}
n'=|T|+ |(U\cup V)\setminus (U_4\cup V_4)|
&\leq 30h^2\varepsilon n+160 h^3\varepsilon n^{\frac{3}{2}} +3\cdot(6h)^{s+3}\varepsilon^{\frac{s+1}{2}} n^{\frac{s+2}{2}}.
\end{align*}
Let $ \varepsilon= \alpha n^{-\frac{s}{s+1}}$. Then for $s\geq 2$, $\alpha<1$ and $h\leq 2sk$, we have
\begin{align*}
n'&= 30h^2\alpha n^{\frac{1}{s+1}}+160 h^3\alpha n^{\frac{3}{2}-\frac{s}{s+1}}+3\cdot(6h)^{s+3}\alpha^{\frac{s+1}{2}} n\\[5pt]
&\leq \left(30h^2+160h^3+3\cdot(6h)^{s+3}\alpha^{\frac{s-1}{2}}\right)\alpha n\\[5pt]
&\leq 4\cdot (6h)^{s+3}\alpha n\\[5pt]
&\leq 4\cdot (12sk)^{s+3}\alpha n.
\end{align*}
This completes the proof.
\end{proof}

\end{document}